\documentclass[11pt]{article}
\usepackage{times,amsthm,amsmath,hyperref,float,theoremref,enumerate}
\usepackage{bm} 
\usepackage{bbm}
\usepackage{graphicx}
\usepackage{amsmath, amssymb, epsfig}
\usepackage{mathrsfs}
\usepackage{color}
\usepackage{epstopdf}
\usepackage{comment,cite,algpseudocode,pdfpages}
\usepackage[scaled]{helvet} 
\usepackage{fullpage}
\usepackage{lipsum}
\usepackage{tikz}

\let\OLDthebibliography\thebibliography
\renewcommand\thebibliography[1]{
	\OLDthebibliography{#1}
	\setlength{\parskip}{0pt}
	\setlength{\itemsep}{0pt plus 0.3ex}
}

\newcommand{\R}{\mathbb{R}}

\newcommand{\E}{\mathbb{E}}

\theoremstyle{definition}

\newtheorem{theorem}{Theorem}
\newtheorem{proposition}{Proposition}
\newtheorem{lemma}{Lemma}

\newtheorem*{method*}{Method}

\begin{document}
	\title{A Sampling Kaczmarz-Motzkin Algorithm for Linear Feasibility}
	\author{Jes\'us A. De Loera, Jamie Haddock, Deanna Needell}
	\date{}                                           
	
	\maketitle

\begin{abstract}
We combine two iterative algorithms for solving large-scale systems of linear inequalities, the relaxation method of  Agmon, Motzkin et al. and the randomized Kaczmarz method. We obtain a family of algorithms that generalize and extend both projection-based techniques. We prove several convergence results, and our computational  experiments show our algorithms often outperform the original methods.
\end{abstract}

\section{Introduction}

We are interested solving large-scale systems of linear inequalities $ Ax \leq b$. Here $b \in \R^m$ and $A$ an $m \times n$ matrix; the regime  $m\gg n$ is our setting of interest, where iterative methods are typically employed. We denote the rows of $A$ by the vectors $a_1,a_2,\dots,a_m$. It is an elementary fact that the set of all $x \in \R^n$ that satisfy the above constraints is a convex polyhedral region,  which we will denote by $P$.  This paper merges two iterated-projection methods, the relaxation method of  Agmon, Motzkin et al. and the randomized Kaczmarz method. For the most part, these two methods have not met each other and have not been analyzed in a unified framework.  The combination of these two algorithmic branches of thought results in an interesting new family of algorithms which generalizes and outperforms its predecessors. We begin with a short description of these two classical methods.


{\bfseries Motzkin's method.}  The first branch of research in linear feasibility is the so-called 
 \emph{relaxation method} or \emph{Motzkin's method}. It is clear from the literature that this is not
 well-known, say among researchers in machine learning, and some results have been re-discovered several times. E.g., the famous 1958 \emph{perceptron} algorithm \cite{rosenblatt} can be thought of a member of this family of methods;  but the very first relaxation-type algorithm analysis appeared a few years earlier in 1954, within the work of Agmon \cite{agmon}, and Motzkin and  Schoenberg \cite{motzkinschoenberg}. Additionally, the relaxation method has been referred to as the Kaczmarz method with the ``most violated constraint control'' or the ``maximal-residual control'' \cite{CensorRowAction, GreedyKaczmarz, Petra2016}. This method can be described as follows: Starting from any initial point $x_0$, a sequence of points is generated.
If the current point $x_i$ is feasible we stop, else there must be a constraint $a^Tx \leq b$ that is \emph{most violated}. 
The constraint defines a hyperplane $H$. 
If $w_H$ is the orthogonal projection of $x_i$ onto the hyperplane $H$, choose a 
number $\lambda$ (normally chosen between $0$ and $2$),
and the new point $x_{i+1}$ is given by $x_{i+1}=x_i+\lambda(w_H-x_i)$. Figure \ref{fig:relax2} displays the iteration visually.

\begin{figure}[h]
\begin{center}
\includegraphics[scale=0.6]{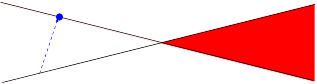} \hskip .3cm \includegraphics[scale=0.6]{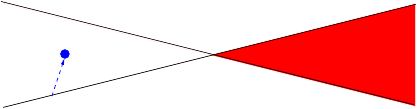} \hskip .3cm \includegraphics[scale=0.6]{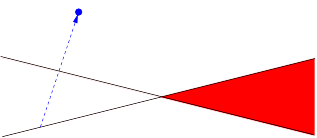} \\
\includegraphics[scale=.25]{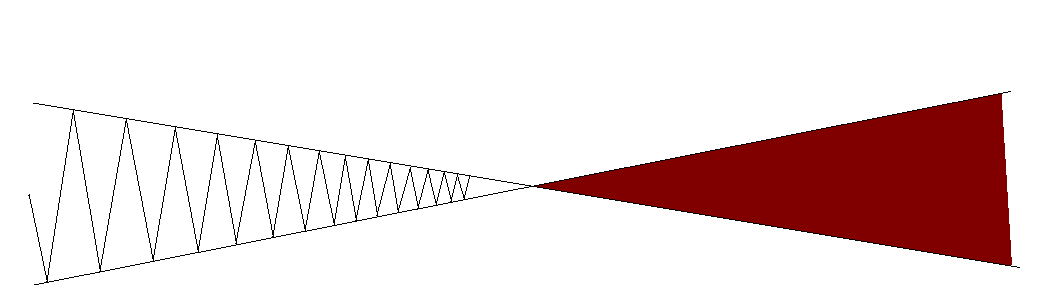}
\end{center}
\caption{three projections with $\lambda=1$, $\lambda <1$ and $\lambda >1$ and a visualization of several steps of the algorithm.}
\label{fig:relax2}

\end{figure}

Many modifications and analyses of this technique have been published since the 1950s, creating an extensive bibliography.  
For example, versions of the relaxation method have suggested various choices of step-length multiplier, $\lambda$ (throughout this paper we consider $\lambda \in (0,2]$),  and various choices for the violated hyperplane.  
The rate of convergence of Motzkin's method depends not only on $\lambda$, but also on the \textit{Hoffman constants} investigated first by
Agmon \cite{agmon} and then later by Hoffmann \cite{hoffman}. If the system of inequalities $Ax \le b$ is feasible, i.e. $P \not= \emptyset$, then there exists Hoffman constants $L_\infty$ and 
$L_2$ so that $d(x,P) \le L_\infty \|(Ax - b)^+\|_\infty$ and $d(x,P) \le L_2 \|(Ax-b)^+\|_2$ for all $x$ (here and throughout, $z^+$ denotes the positive entries of the vector $z$ with zeros elsewhere and $d(x,P)$ the usual distance between a point $x$ and the polytope $P$).  The constants satisfy $L_\infty \le \sqrt{m} L_2$.  When the system of inequalities $Ax \le b$ defines a consistent system of equations $\tilde{A}x = \tilde{b}$ with full column-rank matrix $\tilde{A}$, then the Hoffman constant is simply the norm of the left inverse, $\|\tilde{A}^{-1}\|_2$.
With these constants, one can prove convergence rate results like the following (a spin-off of  Theorem 3 of \cite{agmon} which is easily proven in the style of \cite{leventhallewis}):

\begin{proposition} \label{agmon}
Consider a normalized system with $\|a_i\| = 1$ for all $i=1,...,m$.  If the feasible region $P$ is nonempty then the relaxation method converges linearly: $$d(x_{k},P)^2 \le \bigg(1 - \frac{2\lambda - \lambda^2}{L_\infty^2}\bigg)^k d(x_0,P)^2 \le \bigg(1 - \frac{2\lambda -\lambda^2}{m L_2^2}\bigg)^k d(x_0,P)^2.$$
\end{proposition}

A bad feature of the standard version of the relaxation method using real-valued data is that when the system $Ax \leq
b$ is infeasible it cannot terminate, as there will always be a violated inequality.  In the 1980's the relaxation method was revisited with
interest because of its similarities to the ellipsoid method (see \cite{amaldihauser,betkelp,goffin,telgen} and references therein). One can show that the 
relaxation method is finite in all cases when using rational data, in that it can be modified to detect infeasible systems. In some special cases 
the method gives a polynomial time algorithm (e.g. for totally unimodular matrices\cite{maurrasetal}), but there are also examples of exponential running times 
(see \cite{goffinnonpoly,telgen}).  In late 2010, Chubanov \cite{chubanov}, announced a modification of the traditional relaxation style method, which
gives a \emph{strongly polynomial}-time algorithm in some situations \cite{deloerabasujunod,veghzambelli}. 
Unlike \cite{agmon,motzkinschoenberg}, who only projected onto the original hyperplanes that describe the polyhedron 
$P$, Chubanov \cite{chubanov} projects onto new, auxiliary inequalities which are linear combinations of the input. See 
Figure \ref{fig:ind_adv} for an example of this process.  

\begin{figure}[h]
\begin{center}
\includegraphics[scale=0.42]{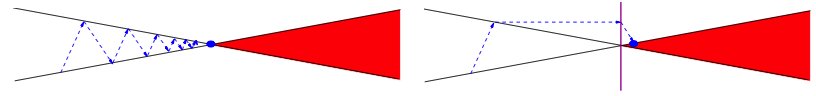}
\end{center}
\caption{Left: Projecting onto original hyperplanes. Right: Projecting onto an induced hyperplane (like those in Chubanov's method).}
\label{fig:ind_adv}
\end{figure}

{\bfseries Kaczmarz method.} The second research branch is that of the 
Kaczmarz method \cite{kaczmarzoriginal,RefWorks:492} which is one of the most popular solvers of overdetermined
systems of linear equations due to its speed and simplicity. Just like Motzkin's, it is an iterative method which consists of a series of alternating
orthogonal projections onto the hyperplanes defined by the system of equations.   The original  Kaczmarz method simply cycles through the equations 
sequentially, so its convergence rate depends on the order of the rows. 
One way to overcome this is to use the equations in a \emph{random order}, rather than sequentially \cite{RefWorks:491,RefWorks:320,RefWorks:514}.  
More precisely, we begin with  $Ax\leq b$, a linear system of inequalities where $A$ is an $m\times n$ matrix with rows $a_i$ and $x_0$ an initial guess.  
For $k = 0, 1, 2, ...$ one defines $$x_{k+1} = x_k - \frac{(\langle a_i, x_k \rangle - b_i)^+}{\|a_i\|^2_2}a_i$$ where $i$ is chosen from $\{1,2,...,m\}$ at random, 
say with probability proportional to $\|a_i\|_2^2$.  Thus, $x_k$ is the projection of $x_{k-1}$ onto the hyperplane $\{x | a_i^Tx = b_i\}$.
Strohmer and Vershynin \cite{strohmervershynin} provided an elegant convergence analysis of the randomized Kaczmarz method for consistent equations.  Later, 
Leventhal and Lewis \cite{leventhallewis} extended the probabilistic analysis from systems of equations to systems of linear inequalities. They focused on giving bounds on the
convergence rate that take into account the numerical conditions captured by the Hoffman constants $L_\infty$ and 
$L_2$.  If one additionally makes use of a projection parameter, $\lambda \not= 1$, you can easily extend the convergence rate in \cite{leventhallewis} to account for this:

\begin{proposition}\label{lewisleventhal}
If the feasible region, $P$,  is nonempty then the Randomized Kaczmarz method with projection parameter $\lambda$ converges linearly in expectation: $$\E[d(x_{k},P)^2] \le \bigg(1 - \frac{2\lambda - \lambda^2}{\|A\|_F^2 L_2^2}\bigg)^k d(x_0,P)^2.$$
\end{proposition}

Note the similarities between Propositions \ref{agmon} and \ref{lewisleventhal}: the convergence rate constants are identical for normalized systems ($\|A\|_F^2 = m$).

The work of Strohmer and Vershynin sparked a new interest in the Kaczmarz approach and there have been many recent developments in the method and its analysis.  Needell \cite{RefWorks:300} extended this work to the case of inconsistent systems of equations, showing exponential convergence down to some fixed \textit{convergence horizon}, see also \cite{wang2015randomized}.  In order to break this convergence horizon, one needs to modify the Kaczmarz method since by design it projects exactly onto a given hyperplane.  Zouzias and Freris \cite{RefWorks:297} analyzed an extended randomized Kaczmarz method which incorporates an additional projection step to reduce the size of the residual.  This was extended to the block case in \cite{needell2015randomized}.  The relation of these approaches to coordinate descent and gradient descent methods has also been recently studied, see e.g. \cite{griebel2012greedy,dumitrescu2014relation,needell2014stochastic,oswald2015convergence,MaConvergence15,HefnyRows15,oswald2015convergence,gower2015randomized}. 

Other variations to the Kaczmarz method include block methods \cite{RefWorks:494,RefWorks:493,RefWorks:298,needelltropp,needellbriskman,RefWorks:562} which have been shown to offer acceleration for certain systems of equations with fast-multipliers.  Other acceleration and convergence schemes focus on sampling selections \cite{agaskar2014randomized,RefWorks:295,needell2014stochastic2,oswald2015random}, projection parameters \cite{RefWorks:478,RefWorks:481,RefWorks:480,RefWorks:319}, adding row directions \cite{RefWorks:477}, parallelized implementations \cite{liu2014asynchronous,avron2014revisiting}, structure exploiting approaches \cite{liu2015accelerated,li2015accelerating}, and the use of preconditioning \cite{gallopoulos2016preconditioners}.  Some other references on recent work include \cite{RefWorks:498,RefWorks:501}

For the most part, it seems that these two branches of research which address the same problems have been developing disjointly from each other. For example, the idea of taking linear combinations of the constraints was first exploited in \cite{chubanov}, but was recently re-discovered and reproduced  for linear equations in \cite{gower2015randomized}, but the authors seem unaware of the optimizers work in  the more general setting of linear inequalities in \cite{chubanov,deloerabasujunod,veghzambelli}. Another example is the manipulation of the projection parameter $\lambda$ \cite{RefWorks:478,RefWorks:481,RefWorks:480,RefWorks:319}. It is a goal of this paper to bridge the separation between these two branches of research that essentially study the same iterative projection procedure. In this paper we explore a family of hybrid algorithms that use elements from both groups of research.

\subsection{Our contribution: the Sampling Kaczmarz-Motzkin method}
Despite the similarity between the Kaczmarz and Motzkin methods (the difference only being in the selection criterion), work on these approaches has remained for the most disjoint.
Our proposed family of methods, which we refer to as the \emph{Sampling Kaczmarz-Motzkin} (SKM) methods, are intended to balance the pros and cons of these related methods.  
Namely, the relaxation method forms iterates whose distance to the polyhedral solution space are monotonically decreasing; however, the time required to choose the most violated hyperplane in each iteration is costly.  Conversely, the Randomized Kaczmarz method has a very inexpensive cost per iteration; however, the method has slow convergence when many of the constraints are satisfied. Our methods will still have a probabilistic choice, like in randomized Kaczmarz,  but make strong use of the  maximum violation criterion within this random sample of the constraints.  Our method is easily seen to interpolate between what was proposed in \cite{leventhallewis} and in \cite{motzkinschoenberg}.

\begin{method*} [SKM method]
Suppose $A \in \mathbb{R}^{m \times n}$, $b \in \mathbb{R}^m$.  Let $x_0 \in \mathbb{R}^n$ be given.  Fix $0 < \lambda \le 2$.  We iteratively construct approximations to a solution lying in $P$ in the following way:  
\begin{enumerate}
\item{Choose a sample of $\beta$ constraints, $\tau_k$, uniformly at random from among the rows of $A$.}
\item{From among these $\beta$ constraints, choose $t_k := \underset{i \in \tau_k}{\text{argmax}} \; a_i^T x_{k-1} - b_i$.}
\item{Define $x_k := x_{k-1} - \lambda \frac{(a_{t_k}^T x_{k-1} - b_{t_k})^+}{\|a_{t_k}\|^2}a_{t_k}$.}
\item{Repeat.}
\end{enumerate}
\end{method*}

\vskip 12pt
\noindent {\bf Remark:} the SKM method with $\beta = m$ recovers the Motzkin relaxation methods, while the SKM method with $\beta = 1$ 
gives a variant of the randomized Kaczmarz method.  We now state our first main result.

\begin{theorem}\label{metatheorem1} 
Let $A$ be normalized so $\|a_i\|^2 = 1$ for all rows $i$.
{If the feasible region $P$ is nonempty then the SKM method with samples of size $\beta$ converges at least linearly in expectation and the bound on the rate depends on the number of satisfied constraints in the system $Ax \leq b$. More precisely, let $s_{k-1}$ be the number of satisfied constraints after iteration $k-1$ and $V_{k-1} = \max \{m - s_{k-1}, m-\beta + 1\}$; then, in the $k$th iteration, 
$$\mathbb{E}[d(x_k,P)^2] \le \bigg(1 - \frac{2\lambda - \lambda^2}{V_{k-1} L_2^2} \bigg) d(x_{k-1},P)^2 \le \bigg(1 - \frac{2\lambda - \lambda^2}{m L_2^2}\bigg)^k d(x_0, P)^2.$$}\label{linearconvergence} 
\end{theorem}

Our second main theoretical result notes that, for rational data, one can provide a \emph{certificate of feasibility} after finitely many iterations of SKM.  
This is an extension of the results by Telgen \cite{telgen} who also noted the connection between relaxation techniques and the ellipsoid method.
To explain what we mean by a certificate of feasibility we recall the length of the binary encoding of a linear feasibility problem with rational data is 
$$\sigma = \underset{i}{\sum}\underset{j}{\sum} \log(|a_{ij}| + 1) + \underset{i}{\sum} \log(|b_i| + 1) + \log nm + 2.$$ 
Denote the maximum violation of a point $x \in \mathbb{R}^n$ as 
$\theta(x) = \max \{ 0, \underset{i}{\max} \{a_i^T x - b_i \} \}. $

Telgen's proof of the finiteness of the relaxation method makes use of the following lemma (which is key in demonstrating that Khachian's ellipsoidal algorithm is finite and polynomial-time \cite{khachiyan}):

\begin{lemma}\label{lem:infeasible} If the rational system $ A x \le b$ is infeasible, then for all $x \in \mathbb{R}^n,$ the maximum violation satisfies $\theta(x) \ge 2 * 2^{-\sigma}$.
\end{lemma}
Thus, to detect feasibility of the rational system $Ax \le b$, we need only find a point, $x_k$ with $\theta(x_k) < 2 * 2^{-\sigma}$; such a point will be called a \emph{certificate of feasibility}.

In the following theorem, we demonstrate that we expect to find a certificate of feasibility, when the system is feasible, and that if we do not find a certificate after finitely many iterations, we can put a lower bound on the probability that the system is infeasible.  
Furthermore, if the system is feasible, we can bound the probability of finding a certificate of feasibility.

\begin{theorem} \label{metatheorem2}  
Suppose $A, b$ are rational matrices with binary encoding length, $\sigma$, and that we run an SKM method on the normalized 
system $\tilde{A} x \le \tilde{b}$ \Big(where $\tilde{a}_i = \frac{1}{||a_i||}a_i$ and $\tilde{b}_i = \frac{1}{||a_i||}b_i$\Big) with $x_0 = 0$.  Suppose
the number of iterations $k$ satisfies
$$k> \frac{4\sigma - 4 - \log n + 2 \log \bigg(\underset{j \in [m]}{\max} ||a_j||\bigg)}{\log \bigg(\frac{mL_2^2}{mL_2^2 - 2 \lambda + \lambda^2}\bigg)}.$$ 
If the system $A x\le b$ is feasible, the probability that the iterate $x_k$ is not a certificate of feasibility is at most

$$\frac{\max ||a_j|| \; 2^{2\sigma-2}}{n^{1/2}} \bigg(1 - \frac{2 \lambda - \lambda^2}{m L_2^2}\bigg)^{k/2},$$

which decreases with $k$.
\end{theorem}

\vskip 12pt

The final contribution of our paper is a small computational study presented in Section \ref{sec:exp}.  
The main purpose of our experiments  is not to compare the running times versus established methods.
Rather, we wanted to determine how our new algorithms  compare with the classical algorithms of Agmon, Motzkin and Schoenberg, and Kaczmarz.  
We examine how the sampling and projection parameters affects the performance of SKM. We try different types of data, but we  
  assume in most of the data that the number of rows $m$ is large, much larger than $n$. The reason is that this is the regime in which the SKM methods 
  are most relevant and often the only alternative. Iterated-project methods are truly interesting in cases where the number of constraints is very large (possibly
  so large it is unreadable in memory) or when the constraints can only be sampled due to uncertainty or partial information. 
  Such regimes arise naturally in applications of machine learning \cite{CAELbook} and in online linear programming (see \cite{agrawaletal} and its references).
Finally, it has already been shown in prior experiments that, for typical small values of $m,n$ where the system
can be read entirely,  iterated-projection methods are not able to compete with the simplex method (see \cite{deloerabasujunod,hoffmanexp}). Here we 
compare our SKM code with MATLAB's interior-point methods and active set methods code. We also compare SKM with another iterated projection method, the block Kaczmarz method \cite{needelltropp}.


\section{Proof of Theorem \ref{metatheorem1}}

We show that the SKM methods enjoy a linear rate of convergence. We begin with a simple useful observation.

\begin{lemma}\label{summationaverage}
Suppose $\{a_i\}_{i=1}^n, \{b_i\}_{i=1}^n$ are real sequences so that $a_{i+1} > a_{i} > 0$ and $b_{i+1} \ge b_{i} \ge 0$.  Then  
$$\underset{i=1}{\overset{n}{\sum}} a_i b_i \ge \underset{i=1}{\overset{n}{\sum}} \bar{a} b_i,
\text{ where } \bar{a} \text{ is the average } \bar{a} = \frac{1}{n}\sum_{i=1}^n a_i.$$
\end{lemma}
\begin{proof}
Note that $\underset{i=1}{\overset{n}{\sum}} a_i b_i = \underset{i=1}{\overset{n}{\sum}} \bar{a} b_i + \underset{i=1}{\overset{n}{\sum}} (a_i - \bar{a})b_i$, so we need only show that $\underset{i=1}{\overset{n}{\sum}} (a_i - \bar{a})b_i \ge 0$, which is equivalent to $\underset{i=1}{\overset{n}{\sum}} (n a_i - \underset{j=1}{\overset{n}{\sum}} a_j)b_i \ge 0$, so we define the coefficients $c_i :=  n a_i - \underset{j=1}{\overset{n}{\sum}} a_j$.  Now, since $\{a_i\}_{i=1}^n$ is strictly increasing, there is some $1 < k < n$ so that $c_k \le 0$ and $c_{k+1} > 0$ and the $c_i$ are strictly increasing.   
Since $\{b_i\}_{i=1}^n$ is non-negative and non-decreasing we have 
\begin{align*}
\underset{i=1}{\overset{n}{\sum}} c_i b_i &= \underset{i=1}{\overset{k}{\sum}} c_i b_i + \underset{i=k+1}{\overset{n}{\sum}} c_i b_i
\ge \underset{i=1}{\overset{k}{\sum}} c_i b_k + \underset{i=k+1}{\overset{n}{\sum}} c_i b_k
= b_k \underset{i=1}{\overset{n}{\sum}} c_i
= 0.
\end{align*}
Thus, we have $\underset{i=1}{\overset{n}{\sum}} a_i b_i = \underset{i=1}{\overset{n}{\sum}} \bar{a} b_i + \underset{i=1}{\overset{n}{\sum}} (a_i - \bar{a})b_i \ge  \underset{i=1}{\overset{n}{\sum}} \bar{a} b_i.$
\end{proof}

\begin{proof}(of Theorem \ref{metatheorem1} )
Denote by $\mathcal{P}$ the projection operator onto the feasible region $P$, and write $s_j$ for the number of zero entries in the residual $(Ax_j - b)^+$, which correspond to satisfied constraints.  Define $V_j := \max\{m-s_j, m-\beta+1\}$. 
Recalling that the method defines $x_{j+1} = x_j - \lambda(A_{\tau_j}x_j - b_{\tau_j})_{i^*}^+ a_{i^*}$ where $$i^* = \underset{i \in \tau_j}{\text{argmax}} \{a_i^T x_j - b_i, 0\} = \underset{i \in \tau_j}{\text{argmax}} (A_{\tau_j}x_j - b_{\tau_j})_i^+,$$ we have 
\begin{align*}
d(x_{j+1},P)^2 &= \|x_{j+1} - \mathcal{P}(x_{j+1})\|^2
\le \|x_{j+1} - \mathcal{P}(x_j)\|^2 
\\&= \|x_j - \lambda(A_{\tau_j}x_j-b_{\tau_j})_{i^*}^+a_{i^*} - \mathcal{P}(x_j)\|^2 
\\&= \|x_j - \mathcal{P}(x_j)\|^2 + \lambda^2((A_{\tau_j}x_j - b_{\tau_j})^+_{i^*})^2 \|a_{i^*}\|^2 
\\&\hspace{2cm}- 2\lambda(A_{\tau_j} x_j - b_{\tau_j})_{i^*}^+ a_{i^*}^T(x_j - \mathcal{P}(x_j)).
\end{align*}
Since $a_{i^*}^T(x_j - \mathcal{P}(x_j)) \ge a_{i^*}^T x_j - b_{i^*}$, we have that
\begin{align}\label{resid}
d(x_{j+1},P)^2 &\le d(x_j, P)^2 + \lambda^2((A_{\tau_j}x_j - b_{\tau_j})_{i^*}^+)^2 \|a_{i^*}\|^2 
\\&\hspace{2cm}- 2\lambda(A_{\tau_j} x_j - b_{\tau_j})_{i^*}^+ (a_{i^*}^T x_j - b_{i^*})\notag\\
&= d(x_j,P)^2 - (2\lambda - \lambda^2)((A_{\tau_j}x_j - b_{\tau_j})_{i^*}^+)^2\notag
\\&= d(x_j,P)^2 - (2\lambda - \lambda^2)\|(A_{\tau_j}x_j - b_{\tau_j})^+\|_\infty^2.
\end{align}
Now, we take advantage of the fact that, if we consider the size of the entries of $(Ax_j - b)^+$, we can determine the precise probability that a particular entry of the residual vector is selected.  Let $(Ax_j - b)^+_{i_k}$ denote the $(k+\beta)$th smallest entry of the residual vector (i.e., if we order the entries of $(Ax_j - b)^+$ from smallest to largest, we denote by $(Ax_j - b)^+_{i_k}$ the entry in the $(k+\beta)$th position).  Each sample has equal probability of being selected, ${m \choose \beta}^{-1}$.  However, the frequency that each entry of the residual vector will be expected to be selected (in Step 3 of SKM) depends on its size.  The $\beta$th smallest entry will be selected from only one sample, while the $m$-th smallest entry (i.e., the largest entry) will be selected from all samples in which it appears.  Each entry is selected according to the number of samples in which it appears and is largest.  Thus, if we take expectation of both sides (with respect to the probabilistic choice of sample, $\tau_j$, of size $\beta$), then  
\begin{align}
\mathbb{E}[\|(A_{\tau_j} x_j - b_{\tau_j})^+\|_\infty^2] &= \frac{1}{{m \choose \beta}} \underset{k=0}{\overset{m-\beta}{\sum}} {\beta - 1 +k \choose \beta-1} ((Ax_j - b)_{i_k}^+)^2 \label{eq1}
\\&\ge \frac{1}{{m \choose \beta}} \underset{k=0}{\overset{m-\beta}{\sum}} \frac{\underset{\ell=0}{\overset{m-\beta}{\sum}} {\beta-1+\ell \choose \beta-1}}{m-\beta+1} ((Ax_j - b)_{i_k}^+)^2 \label{eq2}
\\&=   \underset{k=0}{\overset{m-\beta}{\sum}} \frac{{1}}{m-\beta+1} ((Ax_j - b)_{i_k}^+)^2 \label{eq3}
\\&\ge \frac{1}{m-\beta+1} \min\bigg\{\frac{m-\beta+1}{m-s_j},1\bigg\}\|(Ax_j - b)^+\|_2^2,\label{eq4} 
\end{align}
where \eqref{eq2} follows from Lemma \ref{summationaverage}, because $\{{\beta - 1 + k \choose \beta - 1}\}_{k=0}^{m -\beta}$  is strictly increasing and $\{(Ax_j - b)^+_{i_k}\}_{k=0}^{m-\beta}$ is non-decreasing.  Equality \eqref{eq3} follows from \eqref{eq2} due to the fact that $\underset{\ell=0}{\overset{m-\beta}{\sum}} {\beta-1+\ell \choose \beta-1} = {m \choose \beta}$ which is known as the column-sum property of Pascal's triangle, among other names. Inequality \eqref{eq4} follows from the fact that the ordered summation in \eqref{eq3} is at least $\frac{m-\beta + 1}{m-s_j}$ of the norm of the residual vector (since $s_j$ of the entries are zero) or is the entire residual vector provided $s_j \ge \beta - 1$.
\\\\Thus, we have 
\begin{align*}
\mathbb{E}[d(x_{j+1},P)^2] &\le d(x_j,P)^2 - (2\lambda-\lambda^2)\mathbb{E}[\|(A_{\tau_j} x_j - b_{\tau_j})^+\|_\infty^2]
\\&\le d(x_j,P)^2 - \frac{2\lambda - \lambda^2}{V_{j}}\|(Ax_j - b)^+\|_2^2 
\le \bigg(1 - \frac{2\lambda - \lambda^2}{V_{j} L_2^2} \bigg) d(x_j,P)^2.
\end{align*}
Since $V_{j} \le m$ in each iteration,
$$\mathbb{E}[d(x_{j+1},P)^2] \le \bigg(1- \frac{2\lambda - \lambda^2}{mL_2^2}\bigg) d(x_j,P)^2.$$
Thus, inductively, we get that $$\mathbb{E}[d(x_k,P)^2] \le \bigg(1 - \frac{2\lambda - \lambda^2}{m L_2^2}\bigg)^k d(x_0, P)^2.$$
\end{proof}

Now, we have that the SKM methods will perform at least as well as the Randomized Kaczmarz method in expectation; however, if we know that after a certain point the iterates satisfy some of the constraints, we can improve our expected convergence rate guarantee.  Clearly, after the first iteration, if $\lambda \ge 1$, in every iteration at least one of the constraints will be satisfied so we can guarantee a very slightly increased expected convergence rate.  However, we can say more based on the geometry of the problem.  

\begin{lemma}\label{pointwisecloser}
The sequence of iterates, $\{x_k\}$ generated by an SKM method are pointwise closer to the feasible polyhedron $P$.  That is, for all $a \in P$, $\|x_k - a\| \le \|x_{k-1} - a\|$ for all iterations $k$.
\end{lemma}
\begin{proof}
For $a \in P$, $\|x_k - a\| \le \|x_{k-1} - a\| \text{ for all } k$ since $a \in P \subset H_{t_k} := \{x : a_{t_k}^Tx \le b_{t_k}\}$ and $x_k$ is the projection of $x_{k-1}$ towards or into the half-space $H_{t_k}$ (provided $x_{k-1} \not\in H_{t_k}$, in which case the inequality is true with equality).
\end{proof}

\begin{figure}[h]
\newcommand{\MyPath}{(-2,13/7)--(9/11,16/11)--(-1.32558, -1.04651)--(-2., -0.625)}
\begin{tikzpicture}[scale=1.3]
\clip (-2.7,-1.5) rectangle (4.5,3.8);
\fill[color=gray] \MyPath -- cycle;
\draw (-2,4)--(-2,-3);
\draw (-3,2)--(4,1);
\draw (3,4)--(-3,-3);
\draw (-3,0)--(1,-2.5);
\node at (-1.4,0.3) {$P$};
\node at (-0.4,1) {$a$};
\draw[fill] (-0.2,1.2) circle [radius=0.07];
\draw[fill] (9/11,16/11) circle [radius=0.07];
\node at (1.3,1.15) {$l \in X$};
\draw[fill] (4.25,2) circle [radius=0.07];
\draw[fill] (2.8,3.2) circle [radius=0.07];
\draw[dashed] (4.25,2)--(2.8,3.2);
\draw[fill] (2.5,1.5) circle [radius=0.07];
\draw[dashed] (2.8,3.2)--(2.5,1.5);
\draw[fill] (1.7,2.15) circle [radius=0.07];
\draw[dashed] (2.5,1.5)--(1.7,2.15);
\draw[fill] (1.58,1.5) circle [radius=0.07];
\draw[dashed] (1.7,2.15)--(1.58,1.5);
\draw[dotted] (-0.2,1.2) circle [radius=1.0495];
\draw[dotted] (-0.2,1.2)--(9/11,16/11);
\node at (0.3,1.34) {$r_a$};
\node at (-1.4, -0.75) {$a'$};
\draw[fill] (-1.4, -0.5) circle [radius=0.07];
\draw[dotted] (-1.4,-0.5) circle [radius=2.9564];
\draw[dotted] (-1.4,-0.5)--(9/11,16/11);
\node at (-0.3,0.4) {$r_{a'}$};
\node at (0.8,2.1) {$S(a)$};
\node at (1.15,0) {$S(a')$};
\end{tikzpicture}
\hspace{0.5cm}
\begin{tikzpicture}[scale=1]
\draw (-1,4) -- (0,-3);
\draw[fill] (-1/7,-2) circle [radius=0.07];
\node at (0.45,-2) {$a \in P$};
\draw[fill] (-1.5,3) circle [radius=0.07];
\node at (0.35,3.18) {$l \in X$};
\node at (-1.7,3) {$l'$};
\node at (-1.3,1.5) {$r_a$};
\node at (0,1.68) {$r_a$};
\draw[fill] (-0.25,3.18) circle [radius=0.07];
\draw[dotted] (-1.5,3)--(-1/7,-2);
\draw[dotted] (-0.25,3.18)--(-1/7,-2);
\node at (0.2, -2.8) {$\pi$};
\end{tikzpicture}
\caption{Left: image of $a \in P$, $r_a$ and $S(a)$ and $l \in \underset{a \in P}{\cap} S(a)$ as defined in Lemma \ref{pointconverge}.  Right: image of $l, l' \in X$ contradicting the full-dimensionality of $P$.}
\end{figure}
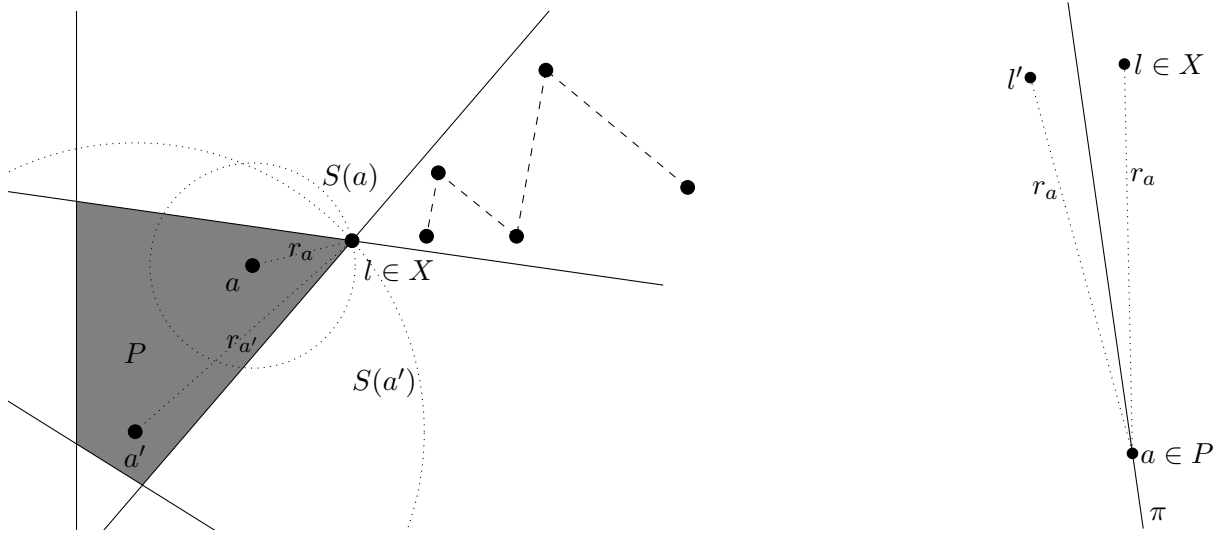

\begin{lemma}\label{pointconverge}
If $P$ is $n$-dimensional (full-dimensional) then the sequence of iterates $\{x_k\}$ generated by an SKM method converge to a point $l \in P$.
\end{lemma}
\begin{proof}
Let $a \in P$.  Note that the limit, $\lim_{k \rightarrow \infty} \|x_k - a\| =: r_a$ exists since $\{\|x_k - a\|\}$ is bounded and decreasing (with probability 1).  Define 
$$S(a) := \{x: \|x-a\| = r_a\} \text{ and } X := \underset{a \in P}{\cap} S(a).$$ 
Note that $X$ is not empty since the bounded sequence $\{x_k\}$ must have a limit point, $l$, achieving $\|l-a\| = r_a$.  
Moreover, suppose there were two such points, $l, l' \in X$.  Define $\pi := \{x : \|l-x\| = \|l' - x\|\}$ to be the hyperplane of points equidistance between $l, l'$.  Then for $a \in P$, we have $l, l' \in S(a)$.  Hence, $a \in \pi$ and we have that $P \subset \pi$, which contradicts the full dimensionality of $P$. Thus $X$ contains only one point, $l$, and it must be a limit point of $\{x_k\}$.  Now, since $\{x_k\}$ is converging to $P$ (with probability one), we must have that $l \in P$.

Now, suppose that $x_k \not\rightarrow l$ (i.e. only a subsequence of $\{x_k\}$ converges to $l$).  Thus, there exists an $\epsilon > 0$ so that for all $K$ there exists $k \ge K$ with $\|x_k - l\| > \epsilon$.  However, there exists a subsequence of $\{x_k\}$ which is converging to $l$, so there must exist some $K_1$ with $\|x_{K_1} - l\| < \epsilon$.  Thus, at some point the sequence $\|x_k - l\|$ must increase, which contradicts Lemma \ref{pointwisecloser}.  Hence, $x_k \rightarrow l$.
\end{proof}

\begin{lemma}\label{violatedconstraints}
Let $l$ be the limit point of the $\{x_k\}$.  There exists an index $K$ so that if $a_j^Tl < b_j$  then $a_j^Tx_k \le b_j$ for all $k \ge K$.
\end{lemma}
\begin{proof}
This is obvious from $x_k \rightarrow l$.
\end{proof}

We would like to conclude with a small ``qualitative'' proposition that indicates there are two stages of behavior of the SKM algorithms. 
After the $K$-th iteration the point is converging to a particular face of the polyhedron. At that
moment one has essentially reduced the calculation to an equality system problem, because the inequalities that define the face 
of convergence need to be met with equality in order to reach the polyhedron.

\begin{proposition} \label{improvedrate}
If the feasible region $P$ is generic and nonempty (i.e., full-dimensional and every vertex satisfies exactly $n$ constraints with equality), then an SKM method with samples of size $\beta \le m - n$ will converge to a single face $F$ of $P$ and all but the constraints defining $F$ will eventually be satisfied.  Thus, the method is guaranteed an increased convergence rate after some index $K$;  for $k \ge K$
$$\mathbb{E}[d(x_k,P)^2] \le \bigg(1 - \frac{2\lambda - \lambda^2}{m L_2^2}\bigg)^K \bigg(1 - \frac{2\lambda - \lambda^2}{(m - \beta + 1) L_2^2}\bigg)^{k-K} d(x_0, P)^2.$$
\end{proposition}

\begin{proof} (of Proposition \ref{improvedrate}))
Since a generic polyhedron is full-dimensional, by Lemma \ref{pointconverge}, we have that the SKM method iterates converge to a point on the boundary of $P$, $l$.  Now, since this $l$ lies on a face of $P$ and $P$ is generic, this face is defined by at most $n$ constraints.  By Lemma \ref{violatedconstraints}, there exists $K$ so that for $k \ge K$ at least $m-n$ of the constraints have been satisfied.  Thus, our proposition follows from Theorem \ref{metatheorem1}.
\end{proof}


\subsection{Proof of Theorem \ref{metatheorem2} }

Now, we show that the general SKM method (when $\lambda \not= 2$) on rational data is finite in expectation.  

We will additionally make use of the following lemma (which is key in demonstrating that Khachian's ellipsoidal algorithm is finite and polynomial-time \cite{khachiyan}) in our proof: 

\begin{lemma}\label{lem:feasible} If the rational system $A x \le b$ is feasible, then there is a feasible solution $\hat{x}$ whose coordinates satisfy $|\hat{x}_j| \le \frac{2^\sigma}{2n}$ for $j = 1,...,n.$
\end{lemma}

Using the bound on the expected distance to the solution polyhedron, $P$, we can show a bound on the expected number of iterations needed to detect feasibility (which does not depend on the size of block selected).

\begin{proof}(of Theorem \ref{metatheorem2})  First, note that if $\tilde{P} := \{x | \tilde{A} x \le \tilde{b}\}$, then $P = \tilde{P}$.  Then, by Lemma \ref{lem:feasible}, if $\tilde{A}x \le \tilde{b}$ is feasible (so $Ax \le b$ is feasible) then there is a feasible solution $\hat{x}$ with $|\hat{x}_j| < \frac{2^\sigma}{2n}$ for all $j = 1, 2, ..., n$ (here $\sigma$ is the binary encoding length for the unnormalized $A, b$).  Thus, since $x_0 = 0$, 

$$d(x_0, P) = d(x_0, \tilde{P}) \le ||\hat{x}|| \le \frac{2^{\sigma-1}}{n^{1/2}}.$$

Now, define $\tilde{\theta}(x)$ to be the maximum violation for the new, normalized system $\tilde{A}x \le \tilde{b}$, $$\tilde{\theta}(x) := \max \{ 0, \underset{i \in [m]}{\max}\; \tilde{a}_i^T x - \tilde{b}_i\} = \max \bigg\{ 0, \underset{i \in [m]}{\max}\; \frac{a_i^T x - b_i}{||a_i||}\bigg\}.$$  By Lemma \ref{lem:infeasible}, if the system $\tilde{A} x \le \tilde{b}$ is infeasible (so $A x \le b$ is infeasible), then $$\tilde{\theta}(x) = \max \{ 0, \underset{i \in [m]}{\max}\; \frac{a_i^T x - b_i}{||a_i||}\} \ge \frac{\max \{ 0, \underset{i \in [m]}{\max}\; a_i^T x - b_i\}}{\underset{j \in [m]}{\max} ||a_j||} = \frac{\theta(x)}{\underset{j \in [m]}{\max} ||a_j||} \ge \frac{2^{1-\sigma}}{\underset{j \in [m]}{\max} ||a_j||}.$$

When running SKM on $\tilde{A}x \le \tilde{b}$, we can conclude that the system is feasible when $\tilde{\theta}(x) < \frac{2^{1-\sigma}}{\underset{j \in [m]}{\max} \; ||a_j||}.$  Now, since  every point of $P$ is inside the half-space defined by $\{x | \tilde{a}_i^Tx \le \tilde{b}_i\}$ for all $i = 1,\cdots, m$,  we have $\tilde{\theta}(x) =  \max \{ 0, \underset{i \in [m]}{\max}\; \tilde{a}_i^T x - \tilde{b}_i\} \le d(x,P).$  Therefore, if $Ax \le b$ is feasible, then

 $$\mathbb{E}(\tilde{\theta}(x_k)) \le \mathbb{E}(d(x_k,P)) \le \bigg(1 - \frac{2\lambda - \lambda^2}{m L_2^2}\bigg)^{k/2} d(x_0, P) \le \bigg(1 - \frac{2\lambda - \lambda^2}{m L_2^2}\bigg)^{k/2}  \frac{2^{\sigma-1}}{n^{1/2}},$$ where the second inequality follows from Theorem \ref{metatheorem1} and the third inequality follows from Lemma \ref{lem:feasible} and the discussion above.

Now, we anticipate to have detected feasibility when $\mathbb{E}(\tilde{\theta}(x_k)) < \frac{2^{1-\sigma}}{\underset{j \in [m]}{\max} \; ||a_j||}$, which is true for $$k > \frac{4\sigma - 4 - \log n + 2 \log \bigg(\underset{j \in [m]}{\max} ||a_j||\bigg)}{\log \bigg(\frac{mL_2^2}{mL_2^2 - 2 \lambda + \lambda^2}\bigg)}.$$  Furthermore, by Markov's inequality (see e.g., \cite[Section 8.2]{sheldon2002first}),
 if the system $Ax \le b$ is feasible, then the probability of not having a certificate of feasibility is bounded: $$\mathbb{P}\bigg(\tilde{\theta}(x_k) \ge \frac{2^{1-\sigma}}{\underset{j \in [m]}{\max}\; ||a_j||}\bigg) \le \frac{\mathbb{E}(\tilde{\theta}(x_k))}{\frac{2^{1-\sigma}}{\underset{j \in [m]}{\max}\; ||a_j||}} < \frac{\bigg(1 - \frac{2\lambda - \lambda^2}{m L_2^2}\bigg)^{k/2}  \frac{2^{\sigma-1}}{n^{1/2}}}{\frac{2^{1-\sigma}}{\underset{j \in [m]}{\max}\; ||a_j||}}$$ $$= \frac{2^{2\sigma-2} \max \; ||a_j||}{n^{1/2}} \bigg(1 - \frac{2\lambda - \lambda^2}{m L_2^2}\bigg)^{k/2}.$$  This completes the proof.
\end{proof}


\section{Experiments}\label{sec:exp}

We implemented the SKM methods in MATLAB \cite{MATLAB:2016} on a 32GB RAM 8-node cluster (although we did not exploit any parallelization), each with 12 cores of Intel Xeon E5-2640 v2 CPUs running at 2 GHz, and ran them on systems while varying the projection parameter, $\lambda$, and the sample size, $\beta$.  We divided our tests into three broad categories: random data, non-random data, and comparisons to other methods.  Our experiments focus on the regime $m\gg n$, since as mentioned earlier, this is the setting in which iterative methods are usually applied; however, we see similar behavior in the underdetermined setting as well.

\subsection{Experiments on random data}\label{sec:randexp}

First we considered systems $Ax \le b$ where $A$ has entries consisting of standard normal random variables and $b$ is chosen to force the system to have a solution set with non-empty interior (we generated a consistent system of equations and then perturbed the right hand side with the absolute value of a standard normal error vector). We additionally considered systems where the rows of $A$ are highly correlated (each row consists only of entries chosen uniformly at random from $[.9,1]$ or only of entries chosen uniformly at random from $[-1, -.9]$) and $b$ is chosen as above.
We vary the size of $A \in \mathbb{R}^{m \times n}$,  which we note in each example presented below. 

In Figure \ref{fig:lambdaexperiments}, we provide experimental evidence that for each problem there is an optimal choice for the sample size, $\beta$, in terms of computation.  We measure the average computational time necessary for SKM with several choices of sample size $\beta$ to reach halting (positive) residual error $2^{-14}$ (i.e. $||(Ax_k-b)^+||_2 \le 2^{-14}$).  Regardless of choice of projection parameter, $\lambda$, we see a minimum for performance occurs for $\beta$ between $1$ and $m$.

\begin{figure}[h]
\begin{center}
\includegraphics[scale=.32]{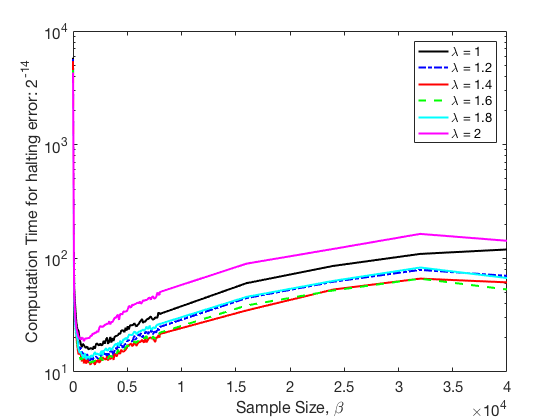}
\includegraphics[scale=.32]{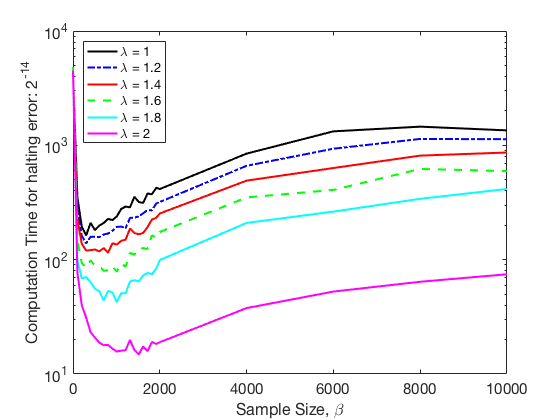}
\end{center}
\caption{Left: Average comp. time for SKM on $40000 \times 100$ Gaussian system to reach residual error $2^{-14}$.  Right: Average comp. time for SKM on $10000 \times 100$ correlated random system to reach residual error.}
\label{fig:lambdaexperiments}
\end{figure}

\begin{figure}
\begin{center}
\includegraphics[scale=.32]{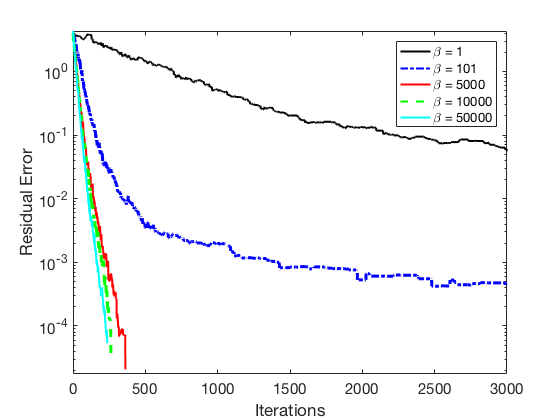}
\includegraphics[scale=.32]{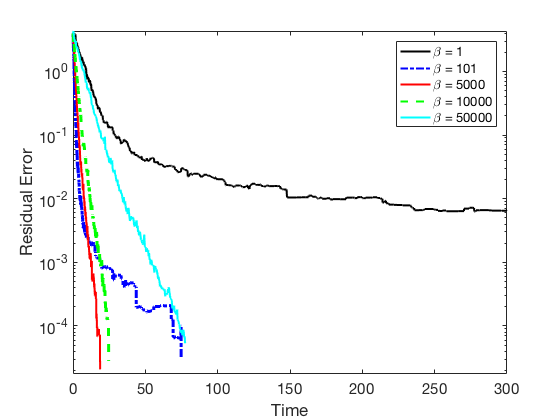}
\end{center}
\caption{Left: Iterations vs. residual error for SKM with various sample sizes on $50000 \times 100$ Gaussian system.  Right: Time vs. residual error.}
\label{fig:errortime}\
\end{figure}

For the experiments in Figures \ref{fig:errortime}, \ref{fig:errortime2}, and \ref{fig:betaexperiments}, we fixed the projection parameter at $\lambda = 1.6$ (for reasons discussed below).  On the left of Figure \ref{fig:errortime2}, we see the residual error decreases more quickly per iteration as the sample size, $\beta$ increases.  However, on the right, when measuring the computational time, SKM with $\beta \approx 5000$ performs best.

In Figure \ref{fig:betaexperiments}, we ran experiments varying the halting error and see that the sample size selection, $\beta$, depends additionally on the desired final distance to the feasible region, $P$.  On the right, we attempted to pinpoint the optimal choice of $\beta$ by reducing the sample sizes we were considering.

Like \cite{strohmervershynin}, we observe that `overshooting' ($\lambda > 1$) outperforms other projection parameters, $\lambda \le 1$. In Figure \ref{fig:lambdaexperiments}, we see that the optimal projection parameter, $\lambda$ is system dependent.  For the experiments in Figure \ref{fig:lambdaexperiments}, we ran SKM on the same system until the iterates had residual error less than $2^{-14}$ and averaged the computational time taken over ten runs.  The best choice of $\lambda$ differed greatly between the Gaussian random systems and the correlated random systems; for Gaussian systems it was $ 1.4 < \lambda < 1.6$ while for correlated systems it was $\lambda = 2$.

Our bound on the distance remaining to the feasible region decreases as the number of satisfied constraints increases.  In Figure \ref{fig:satexperiments}, we see that the fraction of satisfied constraints initially increased most quickly for SKM with sample size, $1 < \beta < m$ and projection parameter, $\lambda > 1$.  On the left, we show that SKM with $\beta = m$ is faster in terms of number of iterations.  However, on the right, we show that SKM with $1 < \beta < m$ outperforms $\beta = m$ in terms of time because of its computational cost in each iteration.

\begin{figure}
\begin{center}
\includegraphics[scale=.32]{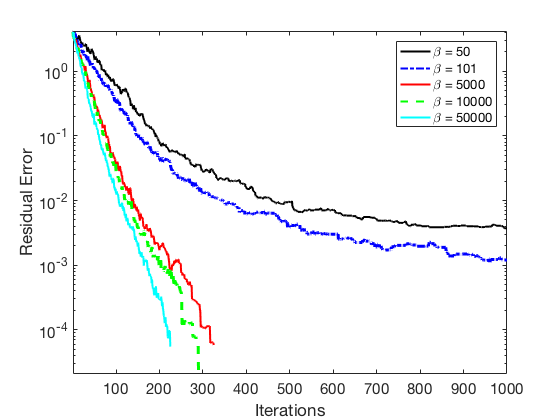}
\includegraphics[scale=.32]{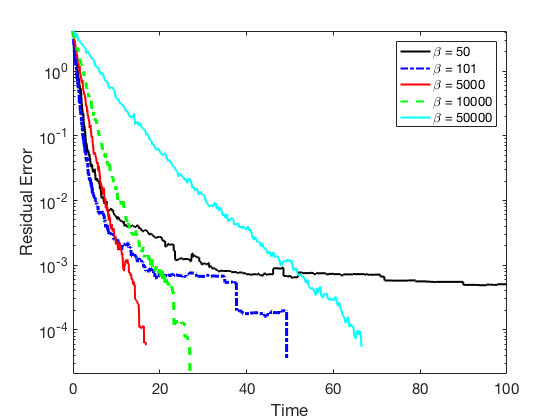}
\end{center}
\caption{Left: Iterations vs. residual error for SKM with sample sizes from $50$ to $m$ on $50000 \times 100$ Gaussian system.  Right: Time vs. residual error.}
\label{fig:errortime2}
\end{figure}

\begin{figure}
\begin{center}
\includegraphics[scale=.32]{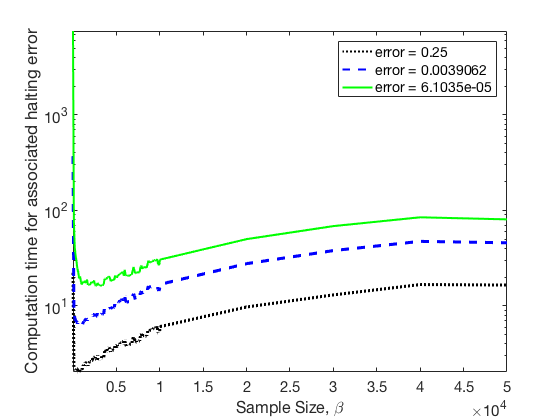}
\includegraphics[scale=.32]{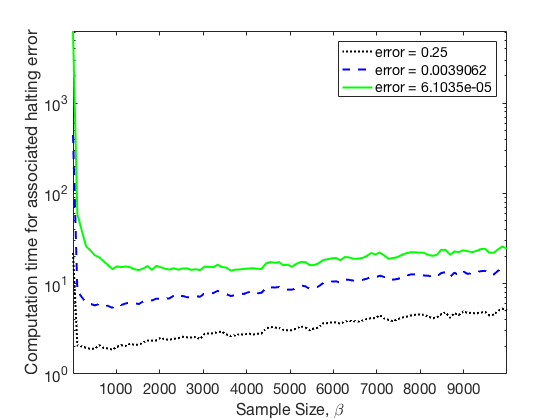}
\end{center}
\caption{Left: Average comp. time for SKM on $50000 \times 100$ Gaussian system to reach various residual errors for $\beta$ between 1 and $m$.  Right: Average comp. time for $\beta$ between 1 and $m/5$.}
\label{fig:betaexperiments}
\end{figure}

\begin{figure}
\begin{center}
\includegraphics[scale=.32]{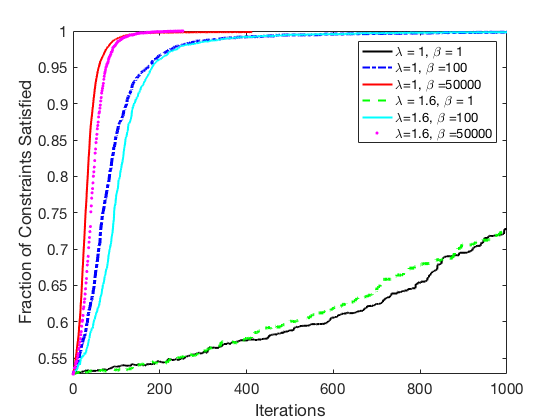}
\includegraphics[scale=.32]{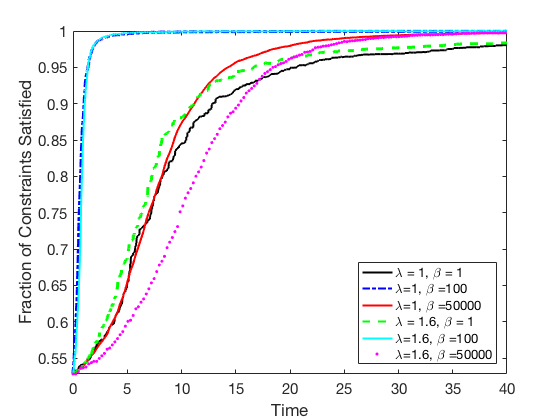}
\end{center}
\caption{Left: Iterations vs. fraction of contraints satisfied for SKM methods on $50000 \times 100$ Gaussian system.  Right: Time vs. fraction of contraints satisfied.}
\label{fig:satexperiments}
\end{figure}

\subsection{Experiments on non-random data}\label{sec:realexp}

We consider next some non-random, non-fabricated test problems: support vector machine (SVM) linear classification instances and feasibility problems equivalent to linear programs arising in well-known benchmark libraries.

We first consider instances that fit the classical SVM problem (see \cite{CAELbook}). 
We used the SKM methods to solve the SVM problem (find a linear classifier) for several data sets from the UCI Machine Learning Repository \cite{UCI}.  The first data set is the well-known Wisconsin (Diagnostic) Breast Cancer data set, which includes data points (vectors) whose features (components) are computed from a digitized image of a fine needle aspirate (FNA) of a breast mass. They describe characteristics of the cell nuclei present in the image.  Each data point is classified as malignant or benign.  The resulting solution to the homogenous system of inequalities, $Ax \le 0$ would ideally define a hyperplane which separates given malignant and benign data points.  However, this data set is not separable.  The system of inequalities has $m=569$ constraints (569 data points) and $n = 30$  variables (29 data features).  Here, SKM is minimizing the residual norm, $||Ax_k||_2$ and is run until $||Ax_k||_2 \le 0.5$.  See Figure \ref{fig:SVMData} for results of SKM runtime on this data set.
\begin{figure}[h]
\begin{center}
\includegraphics[scale=0.32]{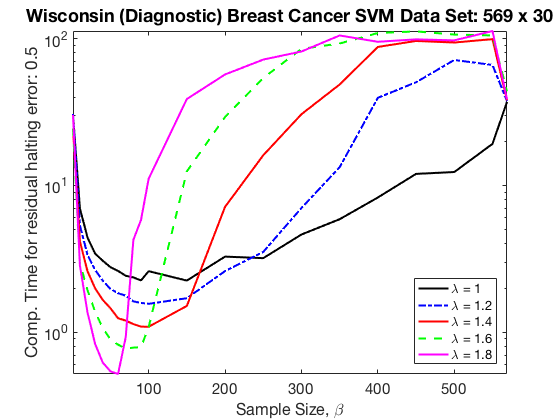}
\includegraphics[scale=0.32]{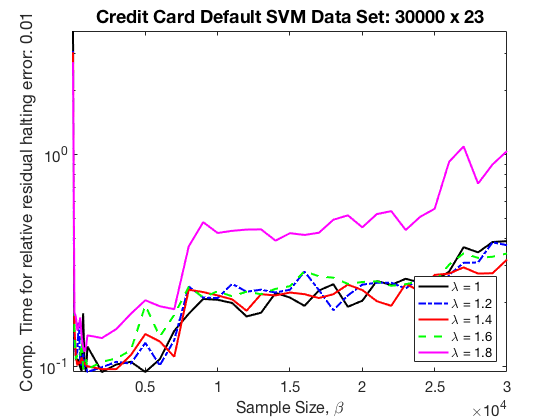}
\end{center}
\caption{Left: Breast Cancer Data SVM.  Right: Credit Card Data SVM.}
\label{fig:SVMData}
\end{figure}

The second data set is a credit card data set, whose data points include features describing the payment profile of a credit card user and the binary classification is for on-time payment or default payment in a billing cycle \cite{creditcard}.  The resulting solution to the homogenous system of inequalities would ideally define a hyperplane which separates given on-time and default data points.  However, this data set is not separable.  The system of inequalities has $m = 30000$ (30000 credit card user profiles) and $n = 23$ (22 profile features).  Here, SKM is run until $||Ax_k||_2 / ||Ax_0||_2 \le 0.01$.  See Figure \ref{fig:SVMData} for results of SKM runtime on this data set.

In the experiments, we again see that for each problem there is an optimal choice for the sample size, $\beta$, in terms of smallest computation time.  
We measure the average computation time necessary for SKM with several choices of sample size $\beta$ to reach the halting (positive) residual error.  Regardless of choice of projection parameter, $\lambda$, we see again that best performance occurs for $\beta$ between $1$ and $m$. Note that the curves are not as smooth as before, which we attribute to the wider irregularity of coefficients, which in turn forces the residual error more to be more dependent on the actual constraints.

We next implemented SKM on several \emph{Netlib} linear programming (LP) problems \cite{Netlib}.  Each of these problems was originally formulated as the LP $\min c^Tx \text{ subject to } Ax = b, \; l \le x \le u$ with optimum value $p^*$.  We reformulated these problems as the equivalent linear feasibility problem $\tilde{A} x \le \tilde{b}$ where $$\tilde{A} = \begin{bmatrix}A\\ -A\\ I\\ -I\\ c^T\end{bmatrix} \text{ and } \tilde{b} = \begin{bmatrix} b\\ -b\\ u\\ -l\\ p^*\end{bmatrix}.$$ See Figures \ref{fig:NetlibData1}, \ref{fig:NetlibData2}, \ref{fig:NetlibData3}, \ref{fig:NetlibData4}, and \ref{fig:NetlibData5} for results of SKM runtime on these problems as we vary $\beta$ and $\lambda$. Once more, regardless of choice of projection parameter, $\lambda$, we see optimal performance occurs for $\beta$ between $1$ and $m$.

It would be possible to handle these equalities without employing our splitting technique to generate inequalities.  This splitting technique only increases $m$ ($||A||_F^2$) and does not affect the Hoffman constant, which is $||\tilde{A}^{-1}||_2$ in this case.  It may be useful to explore such an extension.
\begin{figure}[H]
\includegraphics[scale=0.32]{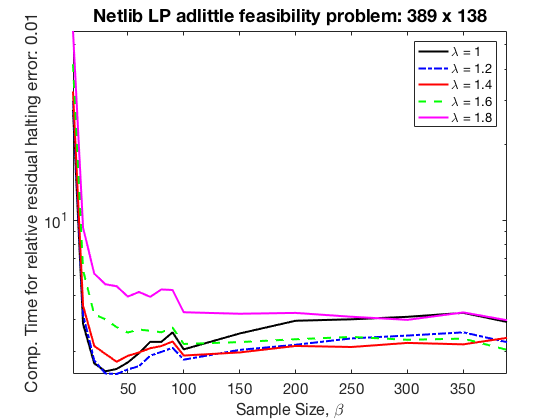} \includegraphics[scale=0.32]{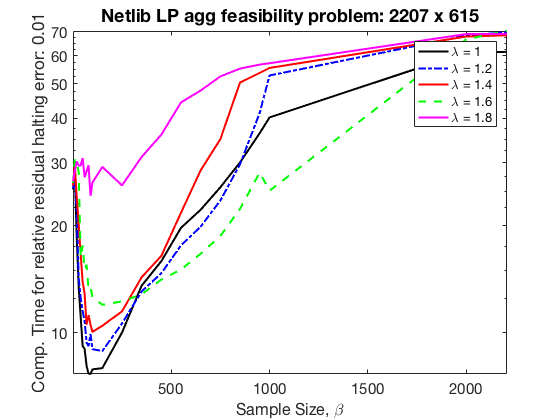}
\caption{Left: SKM behavior for \emph{Netlib} LP adlittle.  Right: SKM behavior for \emph{Netlib} LP agg}
\label{fig:NetlibData1}
\end{figure}
\begin{figure}[H]
 \includegraphics[scale=0.32]{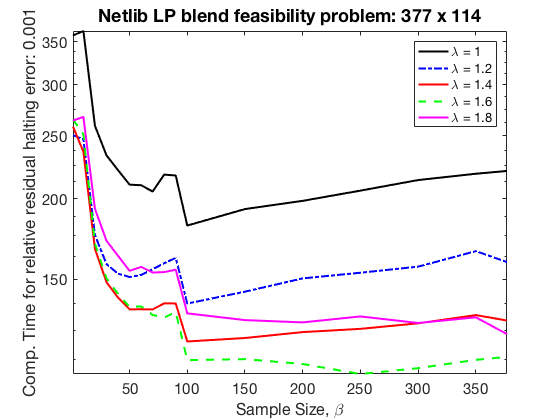} \includegraphics[scale=0.32]{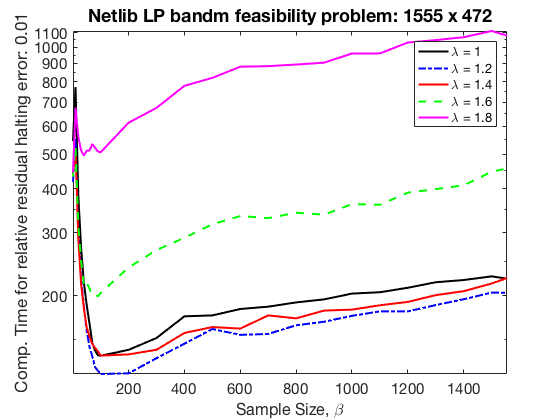}
\caption{Left: SKM behavior for \emph{Netlib} LP blend.  Right: SKM behavior for \emph{Netlib} LP bandm.}
\label{fig:NetlibData2}
\end{figure}
\begin{figure}[H]
\includegraphics[scale=0.32]{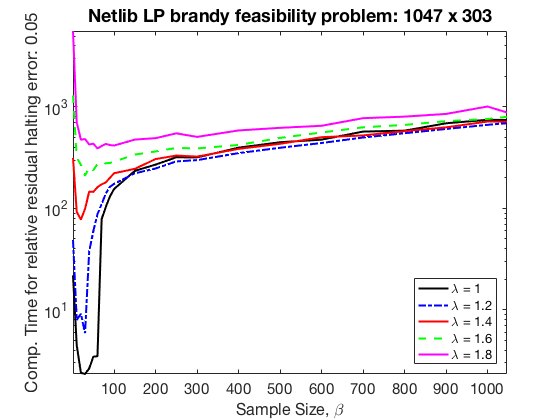} \includegraphics[scale=0.32]{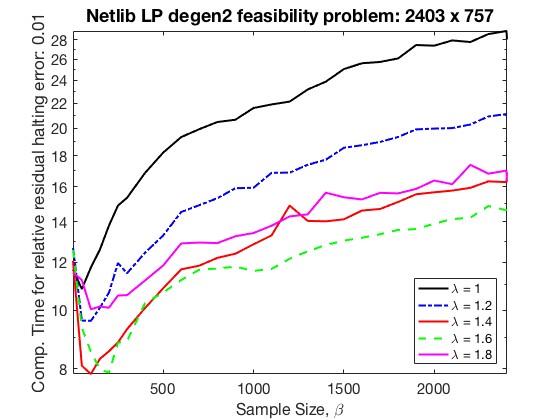}
\caption{Left: SKM behavior for \emph{Netlib} LP brandy.  Right: SKM behavior for \emph{Netlib} LP degen2.}
\label{fig:NetlibData3}
\end{figure}
\begin{figure}[H]
\includegraphics[scale=0.32]{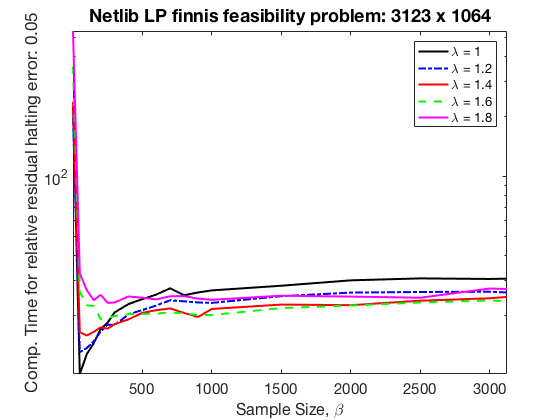} \includegraphics[scale=0.32]{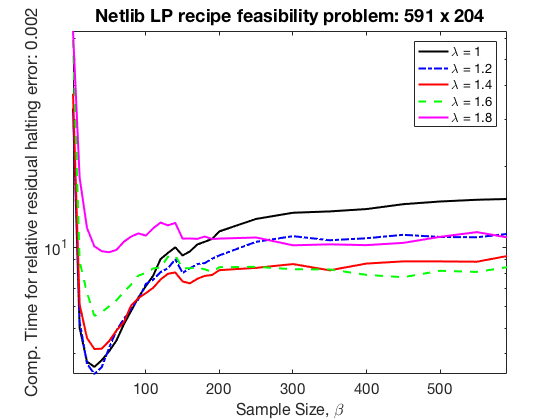}
\caption{Left: SKM behavior for \emph{Netlib} LP finnis.  Right: SKM behavior for \emph{Netlib} LP recipe.}
\label{fig:NetlibData4}
\end{figure}
\begin{figure}[H]
\includegraphics[scale=0.32]{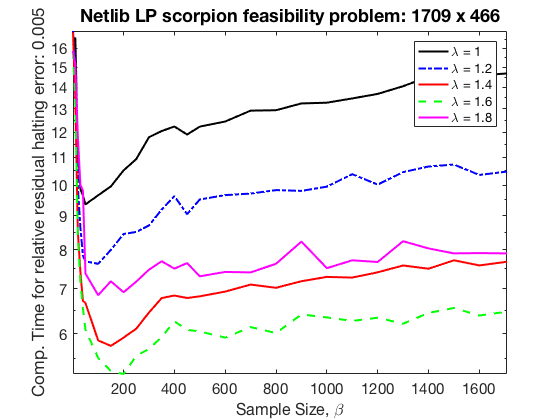}\includegraphics[scale=0.32]{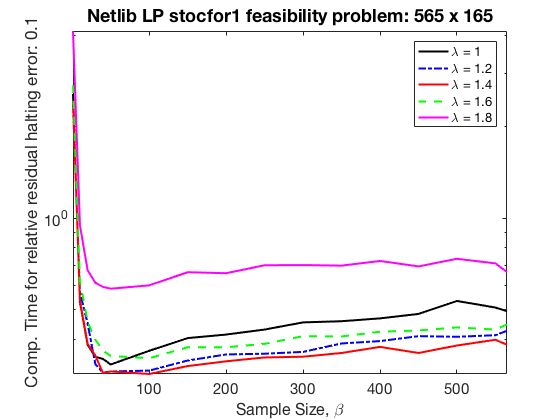}
\caption{Left: SKM behavior for \emph{Netlib} LP scorpion. Right: SKM behavior for \emph{Netlib} LP stocfor1.}
\label{fig:NetlibData5}
\end{figure}

\subsection{Comparison to existing methods} \label{subsec:morecomparison}

\vskip .3cm

In Table \ref{table:lp}, we investigate the performance behavior of SKM versus interior-point and active-set methods on several \emph{Netlib} LPs.  
For fairness of comparison, we gauge our code written in MATLAB versus the MATLAB Optimization Toolbox function \emph{fmincon}. The function \emph{fmincon} allows a user to select either an `interior-point' algorithm or an `active-set' algorithm.

We first used fmincon to solve the feasibility problem as described in Section \ref{sec:realexp} by applying this function to $\min 0 \text{ such that } \tilde{A} x \le \tilde{b}$.  However, the interior-point method and active-set method were mostly unable to solve these feasibility form problems.  The interior-point algorithm was never able to solve feasibility, due to the fact that the system of equations defined by the KKT conditions in each iteration was numerically singular.  Similarly, in most cases, the active-set method was halted in the initial step of finding a feasible point.  For fairness of comparison, we do not list these results.  

In Table \ref{table:lp}, we list CPU timings for the MATLAB interior-point and active-set fmincon algorithms to solve the original optimization LPs ($\min c^Tx \text{ such that } Ax = b, l \le x \le u$), and SKM to solve the equivalent feasibility problem, $\tilde{A}x \le \tilde{b}$, as described in Section \ref{sec:realexp}. Note that this is not an obvious comparion as SKM is designed for feasibility problems, and in principle, the stopping criterion may force SKM to stop near a feasible point, but not necessarily near an optimum. On the other hand, interior point methods and active set methods decrease the value of the objective and simultaneously solve feasibility. 
The halting criterion for SKM remains that  $\frac{\max(\tilde{A}x_k - \tilde{b})}{\max(\tilde{A}x_0 - \tilde{b})} \le \epsilon_{\text{err}}$ where $\epsilon_{\text{err}}$ is the halting error bound listed for each problem in the table.  The halting criterion for the fmincon algorithms is that $\frac{\max(Ax_k - b, l - x_k, x_k - u)}{\max(Ax_0 - b, l-x_0, x_0 - u)} \le \epsilon_{\text{err}}$ and $\frac{c^Tx_k}{c^Tx_0} \le \epsilon_{\text{err}}$ where $\epsilon_{\text{err}}$ is the halting error bound listed for each problem in the table. Each of the methods were started with the same initial point far from the feasible region.
The experiments show our SKM method compares favorably with the other codes.

\begin{table}[tbhp]
\caption{CPU time comparisons for MATLAB methods solving LP and SKM solving feasibility.} 
\label{table:lp}
\begin{tiny}$*$ indicates that the solver did not solve the problem to the desired accuracy due to reaching an upper limit on function evaluations of 100000\end{tiny}
\centering
\begin{small}
\begin{tabular}{ |c|c||c|c|c||c|c|c| } 
 \hline
Problem &Dimensions& Int-Point & SKM & Active-Set & $\epsilon_{\text{err}}$ & SKM $\lambda$ & SKM $\beta$ \\\hline
LP adlittle & $389 \times 138$& 2.08 & 0.29 & 1.85 & $10^{-2}$ & 1.2 & 30 \\\hline
LP agg& $2207 \times 615$ & 109.54* & 20.55 & 554.52* & $10^{-2}$ & 1 & 100 \\\hline
LP bandm& $1555 \times 472$ & 27.21 & 756.71 & 518.44* & $10^{-2}$ & 1.2 & 100 \\\hline
LP blend& $337 \times 114$ & 1.87  & 367.33 & 2.20 & $10^{-3}$ & 1.6 & 250 \\\hline
LP brandy& $1047 \times 303$ & 21.26 & 240.83 & 90.46 & 0.05 & 1 & 20 \\\hline
LP degen2& $2403 \times 757$ & 6.70 & 22.41 & 25725.23 & $10^{-2}$ & 1.4 & 100 \\\hline
LP finnis& $3123 \times 1064$ & 115.47* & 13.76 & 431380.82*  & 0.05 & 1 & 50 \\\hline
LP recipe& $591 \times 204$ &  2.81 & 2.62 & 5.56 & 0.002 & 1.2 & 30 \\\hline
LP scorpion& $1709 \times 466$ & 11.80 & 22.22 & 10.38 & 0.005 & 1.6 & 200 \\\hline
LP stocfor1& $565 \times 165$ & 0.53 & 0.34 & 3.29 & 0.1 & 1.4 & 50\\
 \hline
\end{tabular}
\end{small}
\end{table}

For the experiments in Table \ref{table:lp}, the interior-point method was not able to solve for LP agg and LP finnis before hitting the upper bound on function evaluations due to slow progression towards feasibility.  The active-set method was not able to solve for LP agg, LP bandm and LP finnis before hitting the upper bound on function evaluations due to a very slow (or incomplete) initial step in finding a feasible point.  As mentioned before, the methods were initialized with a point far from the feasible region which may have contributed to the interior-point and active-set methods poor performances.

In Figures \ref{fig:block1} and \ref{fig:block2}, we compare the SKM method to the block Kaczmarz (BK) method (with randomly selected blocks).  Here we solve only systems of linear equations, not inequalities, and we consider only random data as our implemented block Kaczmarz method selects blocks at random.  We see that the performance of the block Kaczmarz method is closely linked to the conditioning of the selected blocks, as the BK method must solve a system of equations in each iteration, rather than one equation as for SKM.  

For the Gaussian random data, the selected blocks are well-conditioned and with high probability, the block division has formed a row-paving of the matrix.  Here we see that BK outperforms SKM.  However, when we consider correlated data instead, the behavior of BK reflects the poor conditioning of the blocks.  In the three included figures, we test with correlated matrices with increasingly poorly conditioned blocks.  If the blocks are numerically ill-conditioned, SKM is able to outperform BK. For systems of equations in which blocks are well conditioned and easy to identify, BK has advantages over SKM.  However, if you are unable or unwilling to find a good paving, SKM can be used and is able to outperform BK.  When BK is used with inequalities, a paving with more strict geometric properties must be found, and this can be computationally challenging, see \cite{needellbriskman} for details. SKM avoids this issue.
\begin{figure}[h]
\includegraphics[scale=0.32]{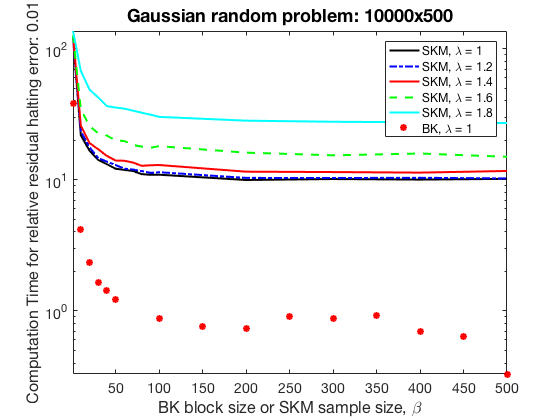}\includegraphics[scale=0.32]{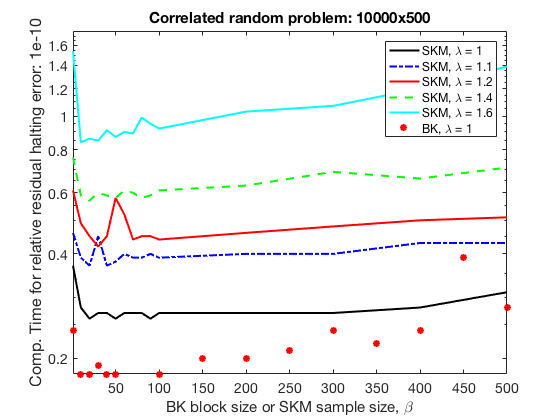}
\caption{Comparison of SKM method runtimes with various choices of sample size, $\beta$ and block Kaczmarz method runtimes with various choices of block size on different types of random systems. Left: Gaussian random system.  Right: Correlated random system with entries chosen uniformly from $[0.9,0.9 + 10^{-5}]$.  }
\label{fig:block1}
\end{figure}

\begin{figure}[H]
\includegraphics[scale=0.32]{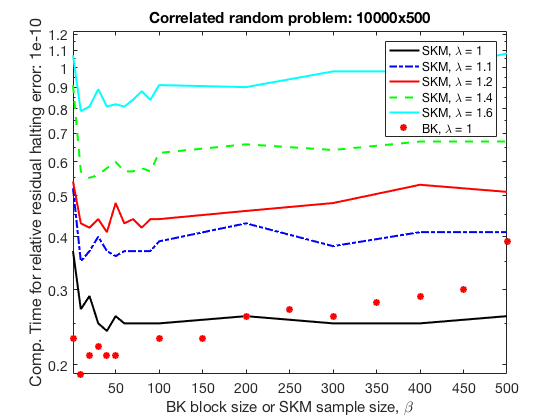}\includegraphics[scale=0.32]{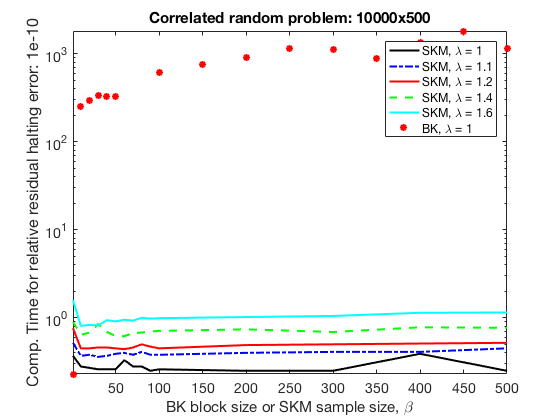}
\caption{Left: Correlated random system with entries chosen uniformly from $[0.9,0.9+10^{-16}]$.  Right: Correlated random system with entries chosen uniformly from $[0.9, 0.9 + 10^{-20}]$.}
\label{fig:block2}
\end{figure}

\section{Remarks about optimal selection of parameters}

\subsection{Choice of $\beta$}  As observed by Theorem \ref{metatheorem1}, the sample size $\beta$ used in each iteration of SKM plays a role in the convergence rate of the method.  
By the definition of $V_{k-1}$ in Theorem \ref{metatheorem1}  and by the bound in Proposition \ref{improvedrate} the choice $\beta = m$ yields the fastest convergence rate.  Indeed, this 
coincides with the classical method of Motzkin; one selects the most violated constraint out of \textit{all} the constraints in each iteration.  However, it is also clear that this choice of $\beta$ 
is extremely costly in terms of computation, and so the more relevant question is about the choice of $\beta$ that optimizes the convergence rate in terms of total computation. 

To gain an understanding of the tradeoff between convergence rate and computation time in terms of the parameter $\beta$, we consider a fixed iteration $j$ and for simplicity choose $\lambda = 1$.  Denote the residual by $r := (Ax_j - b)^+$, and suppose $s$ inequalities are satisfied in this iteration; that is, $r$ has $s$ zero entries.  Write $r_{\tau_j}$ for the portion of the residual selected in Step 3 of SKM (so $|\tau_j|=\beta$).  Then as seen from Equation \eqref{resid} in the proof of Theorem \ref{metatheorem1}, the expected improvement (i.e. $d(x_{j},P) - d(x_{j+1},P)$) made in this iteration is given by $\mathbb{E}\|r_{\tau_j}\|_{\infty}^2$.  Expressing this quantity as in \eqref{eq1} along with Lemma \ref{summationaverage}, one sees that the worst case improvement will be made when the $m-s$ non-zero components of the residual vector are all the same magnitude (i.e. $\mathbb{E}\|r_{\tau_j}\|_{\infty} \geq \frac{1}{m-s}\|r\|_1$).  We thus focus on this scenario in tuning $\beta$ to obtain a minimax heuristic for the optimal selection.  We model the computation count in a fixed iteration as some constant computation time for overhead $C$ plus a factor that scales like $n\beta$, since checking the feasibility of $\beta$ constraints takes time $O(n\beta)$.  We therefore seek a value for $\beta$ that maximizes the ratio of improvement made and computation cost:
\begin{equation}\label{ratio}
\text{gain}(\beta) := \frac{\mathbb{E}\|r_{\tau_j}\|_{\infty}^2}{C + cn\beta},
\end{equation}
when the residual $r$ consists of $m-s$ non-zeros of the same magnitude.   Call the support of the residual $T := \text{supp}(r) = \{i : r_i \ne 0\}$.  Without loss of generality, we may assume that the magnitude of these entries is just $1$.  In that case, one easily computes that 
$$
\mathbb{E}\|r_{\tau_j}\|_{\infty}^2 = \mathbb{P}(T\cap\tau_j \ne \emptyset) = 
\begin{cases}
1 - \frac{\dbinom{s}{\beta}}{\dbinom{m}{\beta}} \approx 1 - \left(\frac{s}{m}\right)^\beta & \text{if $\beta \leq s$},\\
1 & \text{if $\beta > s$},\\
\end{cases}
$$
where we have used Stirling's approximation in the first case.

We may now plot the quantity 
\begin{equation}\label{gain}
\text{gain}(\beta) \approx 
 \frac{1 - \left(\frac{s}{m}\right)^\beta}{C + cn\beta} 
\end{equation} 
as a function of $\beta$, for various choices of $s$.  Figure \ref{fig:betaplot} shows an example of this function for some specific parameter settings.  We see that, as in the experiments of Section \ref{sec:exp}, optimal $\beta$ selection need not necessarily be at either of the endpoints $\beta=1$ or $\beta=m$ (corresponding to classical randomized Kaczmarz and Motzkin's method, respectively).  In particular, one observes that as the number of satisfied constraints $s$ increases, the optimal size of $\beta$ also increases.  This of course is not surprising, since with many satisfied constraints if we use a small value of $\beta$ we are likely to see mostly satisfied constraints in our selection and thus make little to no progress in that iteration.  Again, this plot is for the worst case scenario when the residual has constant non-zero entries, but serves as a heuristic for how one might tune the choice of $\beta$.  In particular, it might be worthwhile to increase $\beta$ throughout the iterations.

\begin{figure}
\begin{center}
\includegraphics[scale=.5]{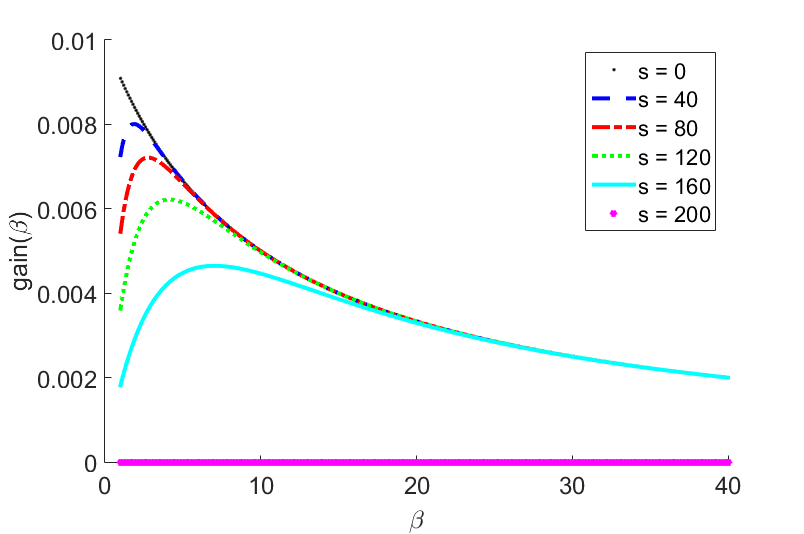}
\end{center}
\caption{The quantity gain($\beta$) as in \eqref{gain} as a function of $\beta$ for various numbers of 
satisfied constraints $s$.  Here we set $m=200$, $n=10$, $c=1$ and $C=100$.  Optimal values of $\beta$ maximize the gain function.}
\label{fig:betaplot}
\end{figure}

\subsection{Choice of $\lambda$}
Additionally, the optimal choice of projection parameter $\lambda$ is system dependent (e.g., for certain systems, one should choose $\lambda = 1$ while for certain full-dimensional systems, one should choose $\lambda > 1$).  
Theoretically, the convergence rate we provided in Theorem \ref{metatheorem1}  depends upon $\lambda$ in a weak way; one would always choose $\lambda = 1$.  However, we see experimentally that overshooting outperforms other choices of $\lambda$.  Additionally, one can easily imagine that for systems whose polyhedral feasible region is full-dimensional, choosing $\lambda > 1$ will outperform $\lambda \le 1$, as eventually, the iterates could `hop' into the the feasible region.  The proof of Proposition \ref{improvedrate} suggests a possible reason why we see this in our experiments.  This proposition is a consequence of the fact that if the method does not terminate then it will converge to a unique face of $P$.  If $\lambda > 1$, then this face cannot be a facet of $P$, as if the method converged to such a face, it would eventually terminate, `hopping' over the facet into $P$.  Thus, for $\lambda > 1$, the number of possible faces of $P$ that the sequence of iterates can converge to is decreased.  Further work is needed before defining the optimal choice of $\lambda$ or $\beta$ for any class of systems.

\subsection{Concluding remarks} 
We have shown SKM is a natural generalization of the methods of Kaczmarz and Motzkin with a theoretical analysis that combines earlier arguments. 
Moreover, compared to these two older methods, the SKM approach leads to significant acceleration with the right choices of parameters.
We wish to note that,  by easy polarization-homogenization of the information (where the hyperplane normals $a_i$ are thought of as points and the solution vector $x$ is a separating plane), one can reinterpret SKM as a type of \emph{stochastic gradient descent} (SGD). Indeed, in SGD one allows the direction to be a random vector whose expected value is the gradient direction; here we generate a random direction that stems from a sampling of the possible increments. More on this will be discussed in a forthcoming article.
 In future work we intend to identify the optimal choices for $\beta$ and $\lambda$ for classes of systems and to connect SKM to Chubanov's style generation of additional linear inequalities that have been successfully used to speed computation \cite{chubanov,deloerabasujunod,veghzambelli}.  All code discussed in this paper is freely available at \url{https://www.math.ucdavis.edu/~jhaddock}.

\section{Acknowledgements} The authors are truly grateful to the anonymous referees and the editor for their many comments and suggestions which have greatly improved this paper.  

\bibliographystyle{siamplain}
\bibliography{bib,bibliolp}

\end{document}